\documentclass[a4paper]{amsart}
\newcommand{\incom}[1]{{\color{blue} { #1}}}
\usepackage[utf8]{inputenc}
\usepackage[pdfusetitle]{hyperref}

\usepackage{amsmath,amsthm,stackengine,scalerel}
\usepackage{amssymb}
\usepackage{mathrsfs}
\usepackage{mathtools}
\usepackage{tikz-cd}
\usepackage{tikz}
\usetikzlibrary{arrows,calc,decorations.pathmorphing,shapes,fit}
\usepackage{enumitem}
\usepackage[capitalise]{cleveref}
\usepackage{csquotes}
\usepackage{fullpage}
\usepackage{multicol}
\usepackage{todonotes}
\newcommand{\margatodo}{\todo[color=pink]}
\newcommand{\margatodoinline}{\todo[inline,color=pink,caption=...]}

\usepackage{faktor}

\newtheoremstyle{mystyle}
{}
{}
{\normalfont}
{}
{\bfseries}
{}
{0.5 em}
{\thmname{#1}\thmnumber{ #2}\normalfont\thmnote{ (#3)}\bfseries.}

\newtheoremstyle{mystyle2}
{}
{}
{\itshape}
{}
{\bfseries}
{}
{0.5 em}
{\thmname{#1}\thmnumber{ #2}\normalfont\thmnote{ (#3)}\bfseries.}

\theoremstyle{mystyle2}
\newtheorem{thm}{Theorem}[section]
\newtheorem*{thm*}{Theorem}
\crefname{thm}{Theorem}{Theorems}
\newtheorem{prop}[thm]{Proposition}
\crefname{prop}{Proposition}{Propositions}
\newtheorem{lem}[thm]{Lemma}
\newtheorem*{lem*}{Lemma}
\newtheorem{kor}[thm]{Corollary}
\newtheorem*{kor*}{Corollary}

\theoremstyle{mystyle}
\newtheorem{defi}[thm]{Definition}
\newtheorem*{defi*}{Definition}
\crefname{defi}{Definition}{Definitions}
\newtheorem{examp}[thm]{Example}
\newtheorem*{examp*}{Example}
\newtheorem{rem}[thm]{Remark}

\newtheorem{fact}[thm]{Fact}

\DeclareMathOperator{\chara}{char}

\DeclareMathOperator{\res}{res}

\renewcommand{\vec}[1]{\overline{#1}}

\newenvironment{claim}[1]{\par\noindent\underline{Claim:}\space#1}{}

\newcommand{\Z}{\mathbb{Z}}
\newcommand{\N}{\mathbb{N}}
\newcommand{\F}{\mathbb{F}}
\newcommand{\Q}{\mathbb{Q}}
\newcommand{\R}{\mathbb{R}}

\let\polishL\L
\newcommand{\Losthm}{\polishL{}o\'s's theorem}

\DeclareMathOperator{\Th}{Th}

\newcommand{\alg}{^\mathrm{alg}}

\newcommand{\U}{\mathcal{U}}

\renewcommand{\L}{\mathcal{L}}
\newcommand{\Lring}{\L_{\mathrm{ring}}}

\newcommand{\Lval}{\L_{\mathrm{val}}}

\renewcommand{\O}{\mathcal{O}}
\newcommand{\m}{\mathfrak{m}}

\newcommand{\barv}{\overline{v}}

\newcommand{\x}{^\times}
\usepackage{xcolor}
\definecolor{sage}{RGB}{107, 200, 17}
\usepackage{verbatim}

\title[Definability via the tilting correspondence]{Arithmetic Definability of Henselian Valuations:\\
Definability via the tilting correspondence}
\author[G.~Alecci]{Gessica Alecci}
\address{Department of Mathematical Sciences
``Giuseppe Luigi Lagrange'', Politecnico di Torino, Italy. \newline
\url{gessica.alecci@polito.it}
}
\author[I.~Hadeg]{Ihsane Hadeg} 
\address{Ruhr-Universit\"at Bochum,
Fakult\"at f\"ur Mathematik,
Universit\"atsstra\ss e 150,
44801 Bochum, Germany.
\newline \url{ihsane.hadeg@ruhr-uni-bochum.de}}

\author[F.~Jahnke]{Franziska Jahnke}
\address{Institute for Mathematical Logic and Foundations, Department of Mathematics and Computer Science,
University of M\"unster,
Einsteinstraße 62,
48149 M\"unster, Germany.\newline
\url{franziska.jahnke@uni-muenster.de}}
\author[M.~Ketelsen]{Margarete Ketelsen} 
\address{Institute for Mathematical Logic and Foundations, Department of Mathematics and Computer Science,
University of M\"unster,
Einsteinstraße 62,
48149 M\"unster, Germany.\newline
\url{margarete.ketelsen@uni-muenster.de}}
\author[I.~Negrini]{Isabella Negrini}
\address{University of Toronto, 40 St. George Street, Room 6290,
Toronto, ON M5S 2E4,
Canada\newline
\url{isabella.negrini@mail.mcgill.ca}}

\begin{document}
\begin{abstract}
    We show that arithmetic definability of henselian valuations is preserved by the tilting correspondence. 
    Moreover, we show that if a perfectoid valuation is arithmetically definable, then
    no parameters are needed.
    We also investigate whether these
    definitions can be chosen uniformly, and discuss the required quantifier complexity.
\end{abstract}
\maketitle

\section{Introduction}
The tilting equivalence is a fundamental tool in arithmetic geometry which allows 
one to ``transfer arithmetic'' between certain fields of mixed characteristic
and their positive characteristic analoga.
More precisely, tilting 
transfers arithmetical data between a perfectoid field of mixed
characteristic and its tilt, an associated perfectoid field of positive characteristic, and vice versa.
Instances of this transfer principle are the Fontaine-Wintenberger Theorem which asserts 
that the absolute Galois group of
a perfectoid field and its tilt are isomorphic as profinite groups (\cite{FontWint}), Scholze's extension of the
tilting equivalence to perfectoid spaces which led to his proof of 
(many cases of) the
weight-monodromy conjecture
in mixed characteristic by reduction to positive characteristic (\cite{scholze2012perfectoid}), 
as well as Andr\'e's proof of the direct summand conjecture which also crucially
uses
perfectoid spaces (\cite{andre}).

These results give rise to the model-theoretic question how much of
the first-order structure of perfectoid fields (in the language of rings or in the language of valued fields, see \cref{sec:mt} for definitions) is preserved by the tilting equivalence. We now give a rough overview of what is known
in this respect so far:
Kartas shows in \cite{kartas2024decidability} that as long as
$K$ is a computable untilt of $K^\flat$ (in an appropriate sense defined in \cite{kartas2024decidability}), 
then decidability of $\textrm{Th}_{\Lval}(K^\flat, v^\flat)$ implies 
decidability of $\textrm{Th}_{\Lval}(K,v)$.
Jahnke and Kartas show in \cite{jahnke-kartas2023beyond} that tilting preserves
elementary equivalence, as well as a version for untilting 
(which requires working over a base field). Moreover, they also obtain a converse to Kartas' results above, showing that decidability of $\textrm{Th}_{\Lval}(K, v)$ 
implies 
decidability of $\textrm{Th}_{\Lval}(K^\flat,v^\flat)$.
A key result of \cite{jahnke-kartas2023beyond} is that -- up to elementary equivalence --
any nonpricipal ultrapower $(K^\mathcal{U}, v^\mathcal{U})$ 
of a perfectoid field $(K,v)$ admits
a coarsening $w$ of $v^\mathcal{U}$ such that $(K^\mathcal{U},w)$ is tame
and such that
$(K^\flat,v^\flat)$ embeds elementarily into the residue field of $(K^\mathcal{U},w)$.
This implies immediately that being $C_i$ in the sense
of Lang (which is an $\Lring$-elementary property) 
 tilts. The converse of this was investigated
by Kartas in \cite{kartas2026perfectoidcitransfer}, but is still 
open in full generality.
In \cite{gambardella2026transferprincipleskatokuzumakiconjecture},
Gambardella and Kartas show that having the strong $C_1^0$-property untilts.
When perfectoid fields are considered in continuous logic (rather than in 
discrete first-order logic),
Rideau-Kikuchi, Scanlon and Simon (\cite{rideaukikuchi2025tiltingequivalencebiinterpretation}) show that the tilting equivalence can be seen
as a bi-interpretation in continuous logic.

The present work adds to the picture described above.
By definition, perfectoid fields come equipped with a rank-1 valuation with
respect to which they are complete. In particular, the valuation is henselian. 
Nontrivial henselian valuations are often so closely related to the arithmetic of the underlying field that they are encoded in it, i.e., that their valuation ring is first-order definable in the language of rings. 
A classical example which goes back to J.~Robinson is that, for $p\neq 2$, the formula
$$\varphi_p(x)\equiv \exists y (y^2=1+px^2)$$
defines the $p$-adic valuation on the $p$-adic numbers, i.e.,
it is satisfied exactly by those $p$-adic numbers which are $p$-adic integers.
In this note, we investigate when a
perfectoid valuation is $\Lring$-definable (with or without parameters; we refer to the latter as being $\emptyset$-definable), and obtain a complete classification as our main
result:

\begin{thm*}[see Theorem \ref{main-thm}]
    	Let \((K,v)\) be perfectoid. Then \(v\) is definable if and only if
    $v$ is $\emptyset$-definable if and only if at least one of the following holds
	\begin{enumerate}
		\item \((K,v)\) is not defectless, 
		\item \(vK\) is not divisible, or
		\item \(Kv\) is not t-henselian of divisible-tame type.
	\end{enumerate}
\end{thm*}
The definitions and notations are explained in \cref{sec:alg} and \cref{def:ttame}.
Since the characterizing conditions in 
\cref{main-thm} are all preserved under tilting, we conclude:
\begin{kor*}[see {\cref{kor:v-def--vflat-def}}]
	Let \((K,v)\) be perfectoid. 
	Then, the following are equivalent:    
    \begin{multicols}{2}
    \begin{itemize}
        \item \(v\) is definable,
        \item $v$ is $\emptyset$-definable,
        \item \(v^\flat\) is definable,
        \item \(v^\flat\) is $\emptyset$- definable.
    \end{itemize}
    \end{multicols}
\end{kor*}
We then comment on how uniformity of definitions can be produced or dispelled arbitrarily,
see \cref{uniform}. In \cref{sec:qf}, we focus on quantifier complexity of
definitions (cf.~\cref{sec:mt} for definitions and notations). Building on work of Anscombe and Fehm (\cite{anscombe-fehm2017characterizing}), we show that $\forall$-$\emptyset$-definability and 
$\exists$-$\emptyset$-definability untilt (\cref{untilt}), but that 
$\forall$-$\emptyset$-definability does not tilt (\cref{nottilt}). Whether 
$\exists$-$\emptyset$-definability tilts remains an open question. Finally, in \cref{sec:classical}, we discuss the classical case of perfectoid fields $(K,v)$ for which
the residue field is an algebraic extension of $\mathbb{F}_p$. The results in this section 
are also straightforward applications of those proven in 
\cite{anscombe-fehm2017characterizing}. We show that if $K$ is not algebraically closed, then $v$ is $\emptyset$-definable, and that if the residue field
is not algebraically closed, $v$ is $\exists$-$\emptyset$-definable.

\bigskip

\textbf{Acknowledgements. }
First and foremost, we thank the organizers of WINE 5 which took place in summer 2025 in Split (Croatia), a lovely event which brought us together and
where the bulk of this work was completed. We also thank our assistant Nora Jahnke
for her cheerful presence and immense patience with us during said event, and Sylvy Anscombe for her excellent 
input into the project, especially regarding \cite{anscombe-fehm2017characterizing} as well as help with various examples. G.A. is a member of the GNSAGA-INdAM. F.J.'s research was 
supported by the Deutsche Forschungsgemeinschaft (DFG, German Research Foundation)
via the ANR-DFG grant AKE Pact (project number 545528554) and a Visiting Fellow position at Merton College, Oxford.
Both F.J. and M.K were supported under Germany's Excellence Strategy EXC 2044/2-390685587, Mathematics M\"unster: Dynamics--Geometry--Structure. 
 
\section{Preliminaries}
\subsection{Algebraic background} \label{sec:alg}
In this subsection, we give an introduction to the notions and tools from 
valuation theory which
we are using throughout. For a more thorough introduction to valued fields,
see \cite{engler2005valued} and \cite{efrat2006valuations}.
\subsubsection*{Notation}
Given a valued field \((K,v)\), we denote its value group by \(vK\), its valuation ring by \(\O_v\), with maximal ideal \(\m_v\), and its residue field by \(Kv\). We denote the residue map \(\O_v\to Kv=\O_v/\m_v\) by \(\res_v\).

\subsubsection{Coarsenings and induced valuations}
In the following we briefly recall the concepts of 
coarsenings and valuation (de)composition. For more details, see \cite[Section~2.3]{engler2005valued}.

Given a field \(K\) and valuations \(v\) and \(w\) on \(K\), we say that \(w\) is a \emph{coarsening} of \(v\) if \(\O_w\supseteq \O_v\).
In this case, we also say that \(w\) is coarser than \(v\), \(v\) is finer than \(w\), or \(v\) is a refinement of \(w\). 

Given a valued field \((K,v)\) and a coarsening \(w\) of \(v\), we find the \emph{induced valuation} \(\overline{v}\) on the residue field \(Kw\):
\[
\overline{v}\colon (Kw)\x \to \Delta,\; \overline{v}(\res_w(x))=v(x),\;  x\in \O_w\x
\]
where \(\Delta\coloneqq \overline{v}((Kw)\x)=v(\O_w\x)\subseteq vK\) is a convex subgroup of \(vK\).
Note that \((Kw)\overline{v}=Kv\).

In this case, we can also write \(v=\barv\circ w\); this is called \emph{valuation decomposition}. 
When representing the valuations by the corresponding places, we obtain the following diagram
\begin{center}
	\begin{tikzcd}
		K \arrow[r, "w"] \arrow[rr, "v", bend left=27] & Kw \arrow[r, "\overline{v}"] & Kv. 
	\end{tikzcd}
\end{center}
Conversely, given a valued field $(K,w)$ and a valuation $u$ on the residue field
$Kw$, one can also form a \emph{valuation composition} \(v=u\circ w\): this is a valuation with valuation
ring $\mathcal{O}_v=\res_w^{-1}(\mathcal{O}_u)$ and in particular a refinement of $w$.

\subsubsection{Defectless valuations}
Let \((K,v)\) be a valued field and let \(K \subseteq L\) be a finite field 
extension. Then $v$ extends to finitely many valuations $w_1, \dotsc, w_r$ on $L$
(with $r \leq [L:K]$, see \cite[Theorem 3.2.9]{engler2005valued}).
Recall that any extension of valued fields \((K,v)\subseteq (L,w)\) induces embeddings of the value groups \(vK\subseteq Lw\) and of the residue fields \(Kv\subseteq Lw\).
We have the following fact:
\begin{fact}[Fundamental inequality, see {\cite[Theorem 3.3.4]{engler2005valued}}]
For a valued field $(K,v)$ and $L/K$ finite, let $w_1, \dotsc, w_r$ denote the
prolongations of $v$ to $L$. Then 
	\[\sum_{i=1}^r (w_iL:vK)[Lw_i:Kv]\leq [L:K].\]
\end{fact}

We say that \((K,v)\) is \emph{defectless} if equality holds for all finite field 
extensions $L$ in the fundamental inequality.
As henselian valuations extend uniquely to any finite extension, a henselian valuation
$(K,v)$ is hence defectless if for all $L/K$ finite
\[ (wL:vK)[Lw:Kv] = [L:K].\]
holds, where $w$ denotes the unique prolongation of $v$ to $L$. If 
$(K,v)$ is henselian and the fundamental equality fails for
$(L,w)/(K,v)$, we say the extension $(L,w)/(K,v)$ has \emph{defect}. 

\subsubsection{Tame fields}
Tame fields were been introduced and studied by Franz-Viktor Kuhlmann in \cite{kuhlmann2016algebra}. See \cite{kuhlmann2025modeltheorytamevalued} for a survey on recent developments.

\begin{defi}
	A henselian valued field \((K,v)\) is \emph{tame} if
	\begin{enumerate}
		\item \((K,v)\) is defectless,
		\item \(Kv\) is perfect, and
		\item \(vK\) is \(p\)-divisible, if \(\chara(Kv)=p>0\).
	\end{enumerate}
\end{defi}

In the following, we will often encounter tame valued fields with \emph{divisible value group}, i.e., for each \(\gamma\in vK\) and \(n\in \N\), there is \(\delta\in vK\) such that \(n\delta=\gamma\).

\begin{examp}
    Let \(k\) be a perfect field of positive characteristic and let \(\Gamma\) be a \(p\)-divisible ordered abelian group.
    Then, the \emph{Hahn series} field \((k(\!(\Gamma)\!),v_\Gamma)\) is a tame valued field.
    The Hahn field construction provides a canonical way to obtain maximally valued fields of equal characteristic with a given value group and residue field.
    For more details see \cite[Section~2.8]{efrat2006valuations}.
\end{examp}

\subsubsection{Perfectoid fields and tilting}
We now give the definitions for the key objects in this paper: perfectoid
fields and their tilts. For more details, see
 \cite{scholze2012perfectoid} and \cite{lurie2018fargues}.
\begin{defi}
	A valued field \((K,v)\) of residue characteristic \(\chara(Kv)=p>0\) is \emph{perfectoid} if all of the following hold:
	\begin{enumerate}
		\item $vK$ has rank 1, i.e. \(vK\leq \R\), and \((K,v)\) is complete, 
		\item \(\O_v/p\O_v\) is \emph{semi-perfect}, that is the Frobenius map 
        \(\O_v/p\O_v\to \O_v/p\O_v, x\mapsto x^p\) is surjective, and
		\item \((K,v)\) is not discretely valued, that is \(vK\ncong\Z\).
	\end{enumerate}
\end{defi}
Examples of perfectoid fields include Hahn series fields $(k(\!(\Gamma)\!),v_\Gamma)$
where $k$ is a perfect field and $\Gamma \leq \mathbb{R}$ is $p$-divisible,
as well as more generally any rank-$1$ complete tame field (\cite[see the remarks on p.~2695 and Theorem 1.2]{kuhlmann-rzepka2023valuation}). More generally however,
perfectoid fields are necessarily henselian but need not be defectless, e.g., the completions of 
$(\mathbb{F}_p(\!(t)\!)^\mathrm{perf},v_t)$ and $(\mathbb{Q}_p(p^{1/p^{\infty}}),v_p)$
are perfectoid but not defectless.

Crucially, to any perfectoid field $(K,v)$ one can associate its tilt $(K^\flat,v^\flat)$, a perfectoid field of positive characteristic. 

\begin{defi}
    Let \((K,v)\) be perfectoid. 
    There is an inverse system of rings
    \[
    \ldots \overset{x\mapsto x^p}{\longrightarrow} \O_v/p\O_v \overset{x\mapsto x^p}{\longrightarrow} \O_v/p\O_v.
    \]
    We define
    \[
    \O_v^\flat \coloneqq \varprojlim_{x\mapsto x^p} \O_v/p\O_v.
    \]
    Now, set \(K^\flat \coloneqq \operatorname{Frac}(\O_v^\flat)\), and let \(v^\flat\) be the valuation on \(K^\flat\) with valuation ring \(\O_v^\flat\) and value group \(vK\).
    The valued field \((K^\flat,v^\flat)\) is called the \emph{tilt} of \((K,v)\).
\end{defi}

As mentioned in the introduction, the
arithmetic of $(K,v)$ and $(K^\flat,v^\flat)$ are `similar'. We only recount
the few basic instances of this similarity that we need for our purposes:
\begin{fact} Let \((K,v)\) be a perfectoid field. \label{fact:perfectoid}
    \begin{enumerate} 
        \item If $\chara(K)=p>0$, then \((K^\flat,v^\flat)\cong (K,v)\).
        \item If $\chara(K)=0$, then \((K,v)\) is a perfectoid field of positive characteristic and we have
    \begin{enumerate}
        \item \(vK \cong v^\flat K^\flat\),
        \item \(\O_v/p\cong \O_{v^\flat}/\varpi\) for some element \(\varpi\in \O_{v^\flat}\), called \emph{pseudouniformizer}, 
        \item \(Kv\cong K^\flat v^\flat\), and
    \end{enumerate}
            \item \((K,v)\) is defectless if and only if \((K^\flat,v^\flat)\) is defectless.
    \end{enumerate}
\end{fact}
\begin{proof}
    (1) and (2) are part of {\cite[Lemma 3.4]{scholze2012perfectoid}}. (3)
follows from \cite[Corollaire 3.3.4]{Wint} -- see also 
\cite[Remark 6.2.5]{jahnke-kartas2023beyond} for the precise statement.
\end{proof}

\subsection{Model-theoretic background}
In this section, we give a very brief introduction into the model-theoretic 
concepts which appear throughout the paper. For more details, we refer the reader to
\cite{hodges1993model} and \cite{hodges1997shorter}.
Throughout this section, let $\L$ be a first-order language and $M$ and $N$ and $\L$-structures.
We write $M \equiv_\L N$ if $M$ and $N$ are elementarily equivalent as $\L$-structures,
i.e., the same $\L$-sentences hold in $M$ and $N$. If $M$ is an $\L$-substructure of 
$N$, we
write $M \preceq_\L N$ in case the given inclusion of $M$ into $N$ is
elementary, that is for all $\L$-formula $\varphi(\vec{x})$, where $\vec{x}$ is 
finite tuple of variables of length $|\vec{x}|$ and all $\vec{m} \in M^{|\vec{x}|}$,
we have
$$M \models \varphi(\vec{m}) \Longleftrightarrow N \models \varphi(\vec{m}).$$
We tend to drop the subscript $\L$ if the language is clear from the context.
For an $\L$-theory $T$, i.e., whenever $T$ is a set of $\L$-sentences, we write
$M \models T$ to indicate that any $\psi \in T$ holds true in $M$.

\label{sec:mt}
\subsubsection{Definability} \label{subsec:def}
We say that a subset $A\subseteq M^n$ is {\em definable} if there is an 
$\mathcal{L}$-formula $\varphi(\vec{x}, \vec{y})$, where $\vec{x}$ and $\vec{y}$ are finite tuples of variables of length $|\vec{x}|=n$ and $|\vec{y}|$
respectively, and some $\vec{b} \in M^{|\vec{y}|}$, such that for any 
$\vec{a}\in M^n$, we have
$$
 \vec{a}\in A \Longleftrightarrow M\models\varphi(\vec{a},\vec{b}),
$$
holds,
in which case we also say that $\varphi(\vec{x},\vec{b})$ defines $A$. We write 
$\varphi(M^n, \vec{b})$ for the set
of elements of $M^n$ which satisfy $\varphi(\vec{x},\vec{b})$.

For a subset $B$ of $M$ we call $\varphi(\vec{x},\vec{b})$ a $B$-formula if 
$\vec{b} \in B^{|\vec{y}|}$ and call $\vec{b}$ the parameters needed to define the 
set
$\varphi(M^{|\vec{x}|},\vec{b})$.
Moreover, we call $\varphi(\vec{x},\vec{y})$ an $\exists$-formula (resp.~$\forall$-formula etc.) if
it is equivalent to a formula in prenex normal form with precisely these quantifiers.
We then say that $A$ is $B$-definable, $\exists$-$B$-definable etc.~if there exists a corresponding formula defining $A$. If $B=\{b\}$ is a singleton, we use $B$-definability and $b$-definability synonymously. A crucial case for us is when
$B = \emptyset$, in which case we say that the corresponding set is defined without parameters.

\subsubsection{Ultrapowers, \polishL{}o\'s and Keisler--Shelah}
The ultrapower construction is a very common and powerful tool.
We refer the reader to \cite[Section~8.5]{hodges1997shorter} for an introduction to filters, ultrafilters, ultrapowers and ultraproducts.
Here, we just state two very useful theorems: \Losthm{}, and the Keisler--Shelah isomorphism theorem.

	Let \(\U\) be an ultrafilter on some index set.
	We denote the \emph{ultrapower} of \(M\) by \(M^\U\):
	
	The underlying set is \(M^\U=\faktor{\prod_{i\in I} M}{\sim_\U}\), where \(\sim_\U\) is the equivalence relation given by 
	\[
	(a_i)_{i\in I}\sim_\U (b_i)_{i\in I} \qquad \Longleftrightarrow \qquad \{i\in I: a_i=b_i\}\in \U.
	\]
	The $\L$-structure on $M$ induces a natural $\L$-structure on $M^\U$, see 
    \cite[Section 8.5]{hodges1997shorter} for details.
	The \emph{diagonal embedding} of $M$ into $M^\U$ is given by
	\begin{align*}
	\iota_M\colon M&\hookrightarrow M^\U\\
	a &\mapsto \left[(a)_{i\in I}\right].
	\end{align*}

The diagonal embedding is an embedding of $\L$-structures, and it is even an elementary 
embedding:
\begin{thm}[\polishL{}o\'s's theorem, see {\cite[Corollary~8.5.4]{hodges1997shorter}}]
	\label{thm:los}
	Let \(M\) be a structure in some language \(\L\) and let \(\U\) be an ultrafilter on some index set \(I\).
	The diagonal embedding \(\iota_M\colon M\hookrightarrow M^\U\) is elementary. In particular, \(M\equiv_\mathcal{L} M^\U\).
\end{thm}
The second key theorem we need about ultraproducts is the Keisler--Shelah Theorem:
\begin{thm}[Keisler--Shelah isomorphism theorem, see {\cite[Theorem~8.5.10]{hodges1997shorter}}]
	\label{thm:Keisler-Shelah}
	Let \(M\) and \(N\) be structures in some language \(\L\), and assume \(M\equiv_\mathcal{L} N\).
	Then there is an ultrafilter \(\U\) on some index set \(I\) such that
	\[
	M^\U\cong N^\U.
	\]
\end{thm}

\subsubsection{Some model theory of valued fields}
In this paper, we consider (valued) fields in the language
\(\Lring=\{0,1,+,-,\cdot\}\) of rings and in the language
\(\Lval=\Lring\cup \{\O\}\) of valued rings, 
where $\O$ is a unary relation symbol interpreted
as the valuation ring. Given fields $K$ and $L$, we often write $K \equiv L$ as
a shorthand for $K\equiv_{\Lring}L$, and when $(K,v)$ and $(L,w)$ are valued fields 
we similarly write $(K,v) \equiv (L,w)$ for $(K,v)\equiv_{\Lval}(L,w)$.

We say that a valuation $v$ on $K$ is {\em definable} if $\mathcal{O}_v$ is a
definable subset of $K$ in the language $\mathcal{L}_{\rm ring}$ of rings, i.e.,
there is an $\Lring$-formula $\varphi(x,\vec{y})$ and a tuple $\vec{c} \in K^{|\vec{y}|}$ 
such that $\mathcal{O}_v = \varphi(K,\vec{c})$
in which case we also say that $\varphi(x,\vec{c})$ defines $v$.
If a valuation is $\emptyset$-definable, then the same formula defines a valuation $w$
in any $L \equiv_{\Lring} K$, and for any such $L$ and $w$ we have 
$(K,v) \equiv_{\Lval} (L,w)$. In particular, properties like henselianity carry over
from $(K,v)$ to $(L,w)$. See \cite{FehmJahnke} for a survey on definable henselian valuations.


\section{(Un)titling Definability}
In this section, we show our main result (\cref{main-thm}): if $(K,v)$ is a perfectoid field of
mixed characteristic $(0,p)$ and its tilt is $(K^\flat,v^\flat)$, then $\mathcal{O}_v$ is an $
\Lring$-definable subset of
$K$ if and only if $\mathcal{O}_{v^\flat}$ is an $
\Lring$-definable subset of
$K^\flat$. We also show that if $v$ (or equivalently $v^\flat$) is 
$\Lring$-definable, then no parameters are needed, i.e., it is $\emptyset$-definable.
Roughly speaking, a henselian valuation can be definable due to its value group,
its residue field or sometimes even using an extension with defect. 
Since the tilting equivalence
preserves value group, residue field and the existence of defect extensions, 
we prove that in each of these cases definability is preserved by (un)tilting. Moreover,
we show that for perfectoid fields, definability of a valuation necessarily 
stems from one of the aforementioned cases.

\subsection*{Definability via non-divisible value groups}

Let $(K,v)$ be a perfectoid field.
We start by asserting that
if  the value group $vK$ is not divisible, then $v$ is definable. 
This follows from a Theorem of Hong, see \cite{hong2014definable} (which in turn generalizes earlier work of Koenigsmann,
see \cite[Lemma 3.6]{Koehab}), which we state below. 
While the statement we give is more
precise than Hong's (who makes no statements about uniformity or quantifier complexity),
his proof yields what is stated below. We thus sketch the proof for the convenience
of the reader.
\begin{thm}[{see \cite[Theorem~4]{hong2014definable}}]
    \label{hong} Let $q$ be a prime.
    There are parameter-free $\Lring$-formulae $\varphi_q(x)$ and $\psi_q(x,y)$ such that
    \begin{enumerate}
  
        \item $\psi_q(x,y)$ is a universal formula and $\psi_q(x,b)$ defines the valuation ring $\mathcal{O}_v$ in any
        henselian valued field with $vK \leq \mathbb{R}$ non-discrete, where $b \in K$
        is such that $q \nmid v(b)$, and
        \item $\varphi_q(x)$ defines the valuation ring $\mathcal{O}_v$ in any
        henselian valued field with $vK \leq \mathbb{R}$ 
    and such that $vK$ is not $q$-divisible and non-discrete.
      \end{enumerate}
\end{thm}
\begin{proof}
For $b \in K$ with $q \nmid v(b)$, 
consider the $\Lring$-formula
\[\xi_q (z,b) \equiv \exists x \exists u (z = bx^q \wedge  u^q- u^{q-1} = z)\]
which uses $b$ as a parameter.
Note that for $z \in K$, we have $K \models \xi_q(z,b)$ 
if and only if $z = bx^q$
for some $x \in K$ and $v(z)>0$: If $z = bx^q$, then $q\nmid v(z)$ since $q \nmid v(b)$,
in particular, $v(z) \neq 0$.
If $v(z)>0$, then Hensel's Lemma gives us $u \in K$ with 
$u^q-u^{q-1}=z$. If $v(z)<0$ and $z = u^q-u^{q-1}$, then $v(u)<0$ and so $q \mid v( u^q-u^{q-1}) = v(z)$, a contradiction.

Now, consider the existential $\Lring$-formula
\[
\omega_q(x,b) \equiv \exists u \exists y  \exists z  (x = u^q-u^{q-1} \wedge \xi_q(z,b) \wedge z(y^q-y^{q-1})=x).
\]
We claim that $\omega_q(x,b)$ defines $\m_v$. 
Indeed, if $K\models \omega_q(x,b)$, suppose $v(x)\le 0$ and take $u,y,z \in K$ witnessing 
$K\models \omega_q(x,b)$. Then, we have $$v(x)= v(u^q-u^{q-1})=qv(u)=v(z)+v(y^q-y^{q-1}) \leq 0.$$ 
If $v(y)>0$, this gives $v(z)=qv(u)-(q-1)v(y)<0$, impossible; if $v(y)\le 0$, then $v(z)=qv(u)-qv(y)$, also impossible since $q\nmid v(z)$. 
Hence $v(a)>0$, so $\omega_q(K,b) \subseteq \m_v$. 

Conversely, let $a\in \m_v$. 
Choose $d \in K$ such that for $z\coloneqq bd^q$ and 
$c\coloneqq a/z$ we have $v(c)>0$ (this exists as $vK$ is dense in $\mathbb{R}$). 
The equation $y^q-y^{q-1}-c=0$ admits a root by Hensel's Lemma, and since $v(a)>0$, $u^q-u^{q-1}=a$ also has a root, so $a\in \omega_q(K,b)$. 
Note that the assumption $q \nmid b$ was in fact not required to show 
$\m_v \subseteq \omega_q(K,b)$, this holds for any $b\neq 0$.
Thus $\omega_q(K,b)=\m_v$. 

Since $\omega_q(x,b)$ is an existential $\Lring$-formula, and we have
$$x \in K\setminus \mathcal{O}_v \Longleftrightarrow \exists y (xy=1  \wedge \omega_q(y,b)),$$
also $K \setminus \mathcal{O}_v$ is definable by an existential 
$\Lring$-formula $\chi_q(x,b)$, and hence 
$\mathcal{O}_v$ is uniformly
definable by the universal $\Lring$-formula $\neg \chi_q(x,b)$.
This proves (i).

In order to eliminate the parameter $b$, consider for any $b\in K^\times$
the $\Lring$-formulae 
\(
\xi_{q}(z,b)
\)
and 
\(
\omega_q(x,b)
\)
as before.
We have verified above that if $q \nmid v(b)$, then $\omega_q(K,b)=\m_v$ and 
that $\omega_q(K,b) \supseteq \m_v$ holds for all $b \neq 0$. 
Thus, $\m_v$ is defined by
$$\chi_q(x)\equiv \forall {b} (b = 0 \vee \omega_q(x,b))$$
and 
hence $\mathcal{O}_v$ is also uniformly $\emptyset$-definable in $\mathcal{L}_{\mathrm{ring}}$ by the formula
$$\varphi_q(x) \equiv \neg \exists y (xy=1  \wedge \chi_q(y)),$$
 proving (ii).
\end{proof}

\subsection*{Definability via residue fields}
Akin for the non-divisibility of the value group in the previous section, we now 
discuss properties of the residue field which give rise to the definability of a 
henselian valuation. The key concept here is t-henselianity, first introduced by Prestel
and Ziegler in \cite{prestel-ziegler1978model}, and its variants. While the definition we 
give below is not the original definition by Prestel and Ziegler, these are
equivalent by \cite[Remark 7.11]{prestel-ziegler1978model} and 
\cite[p. 203]{Pralg}.
\begin{defi}
    \label{def:ttame}
    A field \(K\) is \emph{t-henselian} if there is \(L\equiv K\) such that $L$ admits a nontrivial 
    henselian valuation.
    A field \(K\) is \emph{t-henselian of divisible-tame type} if there is \(L\equiv K\) such that \(L\) admits a nontrivial tame valuation with divisible value group.
\end{defi}
There are t-henselian fields which are neither real nor separably closed and which do not admit a nontrivial henselian valuation \cite[p.~338]{prestel-ziegler1978model}.
The definition of being t-henselian of divisible tame-type stems from 
\cite[Definition 4.4]{anscombe2018henselianity}, where also non-henselian examples of such 
kind where constructed.

\begin{examp}
    \label{ex:alg-cl->t-hens-dtt}
    Every algebraically closed field is t-henselian of divisible-tame type.
    Indeed, if \(K\) is algebraically closed, then \(K\equiv K(\!(\Q)\!)\), and the Hahn series valuation on \(K(\!(\Q)\!)\) is tame with divisible value group.
\end{examp}

\begin{rem}
	Let \(K\) be t-henselian of divisible-tame type. Then there is an ultrafilter \(\U\) on some index set \(I\) such that the ultrapower \(K^\U\) admits a nontrivial tame valuation with divisible value group.
	
	Indeed, let \(L\equiv K\) such that \(L\) admits a nontrivial tame valuation \(w\) with divisible value group. 
	By the Keisler-Shelah isomorphism theorem, see \cref{thm:Keisler-Shelah}, there is an ultrafilter \(\U\) on some index set \(I\) such that \(L^\U\cong K^\U\). Since tameness is an $\Lval$-aximatizable property (\cite[Lemma 4.4]{kuhlmann2016algebra}),
	if we consider the ultrapower \((L^\U,w^\U)\coloneqq (L,w)^\U\) of the valued field \((L,w)\), the valuation \(w^\U\) on \(L^\U\cong K^\U\) will be nontrivial and tame with divisible value group by \Losthm{}, see \cref{thm:los}.
\end{rem}

\begin{lem}
	\label{lem:dtame-decompose}
	Let \((K,v)\) be a valued field with \(v=\barv\circ w\).
	Then \(v\) is tame with divisible value group if and only if both  \(\barv\) and \(w\) are tame with divisible value group.
\end{lem}
\begin{proof}
	Note that \(v\) is henselian and defectless if and only if both  \(\barv\) and \(w\) are henselian and defectless, by \cite[Corollary~4.1.4]{engler2005valued} and \cite[Lemma~2.9]{anscombe2024characterizing}. 
	Also the divisibility of the value group is stable under composition as \(\barv(Kw)\leq_{\mathrm{conv}} Kv\) and \(wK=vK/\barv(Kw)\).
	
	For the rest of the proof we may thus assume that \(v\), \(\barv\) and \(w\) are henselian and defectless with divisible value groups and it remains to show that the residue field of \((K,v)\) is perfect if and only if the residue fields of \((K,w)\) and \((Kw,\barv)\) are.
	
	If \((Kw,\barv)\) and \((K,w)\) have perfect residue fields, then so does \((K,v)\) since \(Kv=(Kw)\barv\).
	Conversely, if \(Kv\) is perfect, then so is \((Kw)\barv=Kv\).
	Hence \((Kw,\barv)\) is tame, and tame fields are perfect, see \cite[Lemma~3.1]{kuhlmann2016algebra}.
	Thus \(Kw\) is perfect.
\end{proof}

We are going to use the following special case of Beth's Definability Theorem to prove a definability criterion.

\begin{fact}
	Let \(T\) be an $\Lval$-theory of valued fields. Then, the following are equivalent. \label{Beth}
	\begin{enumerate}
		\item For any \((F,u_1),(F,u_2)\models T\), we have \(\O_{u_1}=\O_{u_2}\).
		\item There is a parameter-free \(\Lring\)-formula \(\psi\) with \(\psi(F)=\O_u\) in any model \((F,u)\) of \(T\).
		In particular, \(v\) is \(\emptyset\)-definable.
	\end{enumerate}
\end{fact}
\begin{proof}
    This follows immediately from Beth's Definability Theorem ({\cite[Theorem~6.6.4]{hodges1993model}}) where we take \(\L=\Lring\), \(\L^+=\Lval=\Lring\cup\{\O\}\supseteq \L\) and \(\varphi(x)\) as the \(\Lval\)-formula \enquote{\(x\in\O\)}.
\end{proof}

We now apply the above criterion to obtain:
\begin{prop}
    \label{prop:defnotdtt}
    Let \((K,v)\) be tame with divisible value group, \(Kv\) not t-henselian of divisible-tame type. Then, \(v\) is $\emptyset$-definable.   
\end{prop}
\begin{proof}
Firstly, we note that the assumptions imply that \(Kv\) is not separably closed. 
Indeed, \(Kv\) is perfect as it is the residue field of the tame field \((K,v)\).
Now, if \(Kv\) were separably closed, it would be algebraically closed, so by \cref{ex:alg-cl->t-hens-dtt} t-henselian of divisible-tame type, contradicting the assumption.

We now show that for every field $F$ and every two valuations $u_1, u_2$ on $F$, if
\[
(F,u_1), (F,u_2) \models \Th(K,v),
\]
then
\[
\mathcal{O}_{u_1} = \mathcal{O}_{u_2}.
\]
As $(F,u_1) \models\Th(K,v)$, $(F,u_1)$ satisfies the same first-order properties as $(K,v)$. 
In particular, $u_1F$ is divisible, $Fu_1$ is not separably closed, and $Fu_1$ is not $t$-henselian of divisible-tame type\footnote{Note that while not being separably closed
and not being $t$-henselian of divisible tame type are 
not $\Lring$-axiomatizable, they are 
preserved under 
$\Lring$-elementary equivalence.}. 
The same holds for $(F,u_2)$. 
Hence $u_1$ and $u_2$ are comparable by \cite[Theorem 4.4.2]{engler2005valued}. 
Without loss of generality, assume that
\[
\mathcal{O}_{u_1} \supseteq \mathcal{O}_{u_2},
\]
i.e.,\ $u_1$ is a coarsening of $u_2$.
Then $u_2$ induces a valuation $\overline{u_2}$ on the residue field $Fu_1$, and we obtain

\begin{center}
	\begin{tikzcd}
		F \arrow[r, "u_1"] \arrow[rr, "u_2", bend left=27] & Fu_1 \arrow[r, "\overline{u_2}"] & Fu_2 
	\end{tikzcd}
\end{center}

Since $(F,u_2) \models \Th(K,v)$, the valuation $u_2$ is tame with divisible value group. 
By \cref{lem:dtame-decompose}, $\overline{u_2}$ is tame with divisible value group as well. 
Moreover, $Fu_1$ is not $t$-henselian of divisible-tame type, so $\overline{u_2}$ must be trivial. Therefore, we conclude $\mathcal{O}_{u_1} = \mathcal{O}_{u_2}$. By \cref{Beth}, it follows that $v$ is $\emptyset$-definable.
\end{proof}

\subsection*{Definability via defect}
Defining valuation rings via properties of the residue field and value group is
a standard technique, initiated in the work of 
Julia Robinson and extensively studied ever since.
This reflects the Ax--Kochen/Ershov philosophy: in sufficiently nice settings 
(e.g., in equicharacteristic $0$), the first-order theory of a valued field is completely
determined by those of its residue field and value group. However, building on recent
work of Kuhlmann and Rzepka, Ketelsen, Ramello and Szewczyk were able to use Galois
extensions of independent defect-type to define henselian valuations (see \cite{ketelsen2024definable} for more details). In this section,
we refine their result to show that in the rank-1 case, these 
definitions can be done without parameters.
In \cite[Corollary 4.14]{ketelsen2024definable}, the authors show that in any henselian valued field $(K,v)$ with $\chara{Kv} = p > 0$ such that \((K,v)\) is not defectless, and such that either
	\begin{enumerate}
		\item[(i)] $K$ is perfect, if $\chara{K} = p > 0$, or
		\item[(ii)] $\O_v/p$ is semi-perfect, if $\chara{K} = 0$
	\end{enumerate}
    holds, some nontrivial coarsening of $v$ is definable, possibly with parameters.
    Since the only defect occuring in
perfectoid fields is of independent-defect type by \cite[Theorem 1.10]{kuhlmann-rzepka2023valuation},
this gives rise to definability of non-defectless perfectoid valuations:

\begin{kor}
    \label{def-from-defect}
	Let \((K,v)\) be a perfectoid field which is not defectless. 
	Then \(v\) is definable (possibly with parameters).
\end{kor}
\begin{proof}
	By \cite[Corollary~4.14]{ketelsen2024definable}, \(v\) has a nontrivial coarsening \(w\) which is definable.
	Since \((K,v)\) is perfectoid, \(v\) is a rank-1-valuation. 
	Thus, there are no proper nontrivial coarsenings of \(v\), so \(w=v\).
\end{proof}

Moreover, Ketelsen, Ramello and Szewczyk show that if $(K,v)$ has mixed characteristic
and satifies the assumptions of \cite[Corollary 4.14]{ketelsen2024definable}, \emph{some} henselian valuations on $K$
is already $\emptyset$-definable \cite[Corollary 4.16]{ketelsen2024definable} (although this is not necessarily a coarsening of $v$), and that in general, parameters are needed
\cite[Proposition 6.14]{ketelsen2024definable}. 

We now want to show a parameter-free version of \cref{def-from-defect}. 
We start by giving a slightly more verbose 
version of {\cite[Theorem~4.11]{ketelsen2024definable}}.

\begin{thm}[see {\cite[Theorem~4.11]{ketelsen2024definable}}]
    \label{thm:defect}
	Let $(K,v)$ be a henselian valued field with $\chara{Kv} = p > 0$ such that \((K,v)\) admits a Galois extension $L$ of degree $p$ with independent defect.
	If \(\chara(K)=0\), we additionally assume that \(\zeta_p\in K\), where \(\zeta_p\) is a primitive \(p\)-th root of unity. Then $K$ admits a nontrivial definable henselian valuation $w$, coarsening $v$. In fact, $w$ can be defined by a 
    formula $\varphi_K(x, c_0,\dots,c_{p-1})$, where $L=K(\theta)$ and the minimal polynomial
    of $\theta$ over $K$ is $X^p + \sum_{i=0}^{p-1}c_iX^i$.
\end{thm}    
\begin{proof}
    Apart from the last sentence, this is exactly the content of \cite[Theorem~4.11]{ketelsen2024definable}. The proof of \cite[Theorem~4.11]{ketelsen2024definable} 
    shows that $w$ is definable by an $\Lring$-formula which uses only $\vec{c}$ as
    parameters (this is made explicit in the last line of the proof).
\end{proof}

In case $vK \leq \mathbb{R}$, we can now eliminate the parameters in 
a similar fashion to the argument \cite[Lemma 3.6]{Koehab} (and 
somewhat similar to those in \cite{Ax} and \cref{hong}):
\begin{thm}
	Let \((K,v)\) be henselian valued field which is not defectless such that \(vK\leq \R\). 
    Assume further that \(\chara(Kv)=p>0\) and that \label{defect-def}
    	\begin{enumerate}
		\item[(i)] $K$ is perfect, if $\chara{K} = p > 0$, or
		\item[(ii)] $\O_v/p$ is semi-perfect, if $\chara{K} = 0$ and
        $\chara(Kv)=p$
	\end{enumerate}
    holds.
	Then \(v\) is \(\emptyset\)-definable. 
    In particular, if $(K,v)$ is perfectoid and $v$ is not defectless then $v$ is 
    $\emptyset$-definable.
\end{thm}
\begin{proof}
Let \((K,v)\) be henselian valued field which satisfies the assumptions of the theorem.
If $Kv=Kv^\mathrm{alg}$, a nontrivial henselian valuation $w$ on $K$ 
is $\emptyset$-definable by
\cite[Theorem 3.10]{JK14a}. That $w$ is in fact a coarsening of $v$ is argued in
\cite[Remark 4.10]{ketelsen2024definable}. As $v$ has no nontrivial coarsenings, we 
conclude that in this case $v$ is $\emptyset$-definable. 

If $Kv$ is not algebraically closed, we seek to apply \cref{thm:defect}, 
similar in spirit to the proof of \cite[Corollary 4.14]{ketelsen2024definable}.
Note that as $Kv$ is not algebraically closed and has positive characteristic, 
the same holds
for the residue field of the prolongation of $v$ to any finite extension of $K$
\cite[Theorem 4.3.5]{engler2005valued}.
By a standard argument (see \cite[Lemma 4.12]{ketelsen2024definable}), $K$ has a finite extension $K'$ such that $(K',v_{K'})$ admits an immediate Galois defect extension $K'(\theta)$ of degree $p$ (where $v_{K'}$ denotes the unique extension of $v$ to $K'$). In case $\chara{K}=0$, we may further
assume $\zeta_p \in K'$ by replacing $K'$ with $K'(\zeta_p)$ if necessary.
In case $\chara{K}=0$, and $p=2$, we moreover replace $K'$ with $K'(\sqrt{-1})$.
Let $f(X) = X^p + \sum_{i=0}^{p-1}c_i X^i \in K'[X]$ be the minimal polynomial of $\theta$
over $K'$.

We now fix $n:=[K':K]$. Note that since $K'$ admits Galois extensions of degree $p$ (we use the shorthand notation $K'\neq K'(p)$ to denote this) and
is $v_{K'}$ is henselian, the valuation $v_{K'}$ is a nontrivial $p$-henselian 
valuation\footnote{A valuation is called $p$-henselian if it extends uniquely to every 
Galois extension of degree $p$, see \cite{Koe95}.}
on $K'$. By \cite[Main Theorem]{JK15}, $K'$ admits a nontrivial $\emptyset$-definable $p$-henselian 
valuation $v_{K'}^p$ which is comparable to every $p$-henselian valuation on $K'$, the so-called
canonical $p$-henselian valuation. 
In fact, 
\cite[Main Theorem]{JK15} asserts that there
is a parameter-free $\Lring$-formula $\psi_p(x)$ which defines the canonical
$p$-henselian valuation $v_F^p$ in any field $F$, assuming $\zeta_p \in F$ if $\chara{F}=0$ and $p\neq 2$ or $\sqrt{-1} \in F$ if $\chara{F}=\chara{Fv_F^p}=0$ and $p=2$.\\

\begin{claim}
    If $F\neq F(p)$ and $u$ is a henselian valuation on 
$F$ with $uF \leq \mathbb{R}$ then $\mathcal{O}_{v_F^p} \subseteq \mathcal{O}_u$. 
\end{claim}
\begin{proof}[Proof of claim]
    As $v_F^p$ is the canonical $p$-henselian valuation, it is comparable to any $p$-henselian valuation on $F$. 
    If $u$ is trivial, we get $\mathcal{O}_{v_F^p} \subseteq \mathcal{O}_u=F$. 
    Otherwise, as $F\neq F(p)$ and $F$ admits a nontrivial $p$-henselian valuation 
    (namely $u$), $v_F^p$ is also nontrivial. As $uF\leq \mathbb{R}$, $u$ has no 
    nontrivial coarsenings, so we conclude $\mathcal{O}_{v_F^p} \subseteq \mathcal{O}_u$.
\end{proof}
In particular, for $F=K'$ we see that $v_{K'}$ is a coarsening of $v_{K'}^p$.
We now combine $n$, $\psi_p(x)$ and $\varphi_{K'}(x, \vec{y})$ (the latter 
from \cref{thm:defect}) to define $v$ on $K$ 
without parameters. 

We first treat the case $\chara{K}=p$, as it is easier notation-wise not to work
in characteristic $0$ and positive characteristic simultaneously.
Consider the class of fields with a named $p$-tuple
\begin{align*}\mathcal{F}:=\{(F, \vec{d}): \ & [F:K]=n,F\neq F(p), \vec{d}=(d_0,\dotsc,d_{p-1})\in F^p, \\
&\varphi_{K'}(F, \vec{d}) \textrm{ is a valuation ring with } \psi_p(F) \subseteq \varphi_{K'}(F, \vec{d})\neq F \}.
\end{align*}
The family $\mathcal{F}$ is uniformly interpretable in $K$: we quantify over those $n$-tuples from $K$ which are the coefficients of irreducible monic degree $n$ polynomials over $K$, since
such polynomials generate extensions degree $n$ extensions $F/K$ -- this allows us
to uniformly interpret the class $\mathcal{F}_n$ of degree $n$ extensions of $K$ (see \cite[Section 1.2.5]{Chatzidakis2026} for details).
Now, we can moreover uniformly
interpret the class of those $F \in \mathcal{F}_n$ (and tuples $\vec{d} \in F$) 
with certain
$\Lring(K)$-definable properties. This means we can single out those 
$F \in \mathcal{F}_n$ with $F \neq F(p)$ as those are exactly the ones admitting some
Artin-Schreier extension, and those
$\vec{d} \in F$ such that $(F, \vec{d})$ belongs to $\mathcal{F}$.

By the claim above, for any $(F,\vec{d}) \in \mathcal{F}$, $\psi_p(x)$ defines a refinement $v_F^p$
of the unique
prolongation $v_F$ of $v$ to $F$ (since $vK \leq \mathbb{R}$ implies $v_FF\leq 
\mathbb{R}$).
As the valuation rings of coarsenings of $v_F^p$ 
are linearly ordered by inclusion, the valuation rings $\varphi_{K'}(F, \vec{d})$
and that of $v_F$ are comparable. Since $(F, \vec{d}) \in \mathcal{F}$ stipulates 
$\varphi_{K'}(F, \vec{d}) \neq F$, we necessarily have that  $\varphi_{K'}(F, \vec{d})
\subseteq \mathcal{O}_{v_F}$ holds. 

Recall that $K'(\theta)/K'$ is an independent defect extension, 
and the minimal polynomial of $\theta$ is 
$X^p + \sum_{i=0}^{p-1}c_i X^i$. By \cref{thm:defect}, we
have for $\vec{c} = (c_0,\dots,c_{p-1})$ that $\varphi_{K'}(x,\vec{c})$ defines
$v_{K'}$. 
In particular, we have
$(K',\vec{c}) \in \mathcal{F}$.
Thus, we conclude
 $$\mathcal{O}_v = \bigcup_{(F, \vec{d}) \in \mathcal{F}} \varphi_{K'}(F, \vec{d}) \cap K$$
where -- due to the uniform interpretation of $\mathcal{F}$ -- 
the right hand side can be expressed a parameter-free $\Lring$-formula.

For $\chara{K}=0$ and $\chara{Kv}=p$, essentially the same argument works, 
after requiring that fields in 
$\mathcal{F}$
contain $\zeta_p$ 
and replacing Artin Schreier extensions by Kummer extensions.
 \end{proof}

\subsection*{Non-definability}
We now show that if a perfectoid valuation is not already definable due to \cref{hong}, \cref{prop:defnotdtt}, or \cref{defect-def}, it is in fact \emph{not} $\Lring$-definable. In order to prove this, we 
need another statement from \cite{ketelsen2024definable}.
This fact in turn relies on quantifier elimination in divisible ordered abelian groups and stable embeddedness of the value group in tame valued fields.

\begin{fact}[{see \cite[Lemma~3.4]{ketelsen2024definable}}]
	\label{fact:dt-no-proper-coars-defble}
	Let \((K,v)\) be tame with divisible value group.
	No proper coarsening of \(v\) can be \(\Lval\)-definable in \((K,v)\).
	In particular, no proper coarsening of \(v\) can be \(\Lring\)-definable in \(K\).
\end{fact}

As a consequence, we obtain the following non-definability result:
\begin{prop}
	\label{prop:notdef}
	Let \((K,v)\) be tame with divisible value group such that \(Kv\) is t-henselian of divisible-tame type. 
	Then, \(v\) is not definable (even with parameters). 
    \end{prop}
\begin{proof}
	Let \(\U\) be an ultrafilter on some index set \(I\) such that \((Kv)^\U\) admits a nontrivial tame valuation \(u\) with divisible value group.
	We consider the ultrapower \((K^\U,v^\U)\coloneqq (K,v)^\U\) of the valued field \((K,v)\). 
	Then, the composition \(w\coloneqq u\circ v^\U\) is tame with divisible value group as \(u\) and \(v^\U\) are both tame with divisible value group.
	
	Now assume for a contradiction that \(v\) on \(K\) is definable, say \(\O_v=\varphi(K,\vec{c})\) for some \(\Lring\)-formula \(\varphi\) and some parameter tuple \(\vec{c}\in K^n\).
	Then, by \Losthm{}, \(\O_{v^\U}=\varphi(K^\U,\vec{c})\), so \(v^\U\) is definable in \(K^\U\).
	This contradicts \cref{fact:dt-no-proper-coars-defble}, as \(v^\U\) is a proper coarsening of the tame valuation \(w\) with divisible value group.
\end{proof}

We now give examples of perfectoid fields $(K,v)$, both in mixed and in positive characteristic, where the valuation $v$ is not $\Lring$-definable, applying Proposition
\ref{prop:notdef}.
\begin{examp}
    \label{ex:non-def}
    Let $k$ be a non-algebraically closed field of positive characteristic
    such that $k \equiv k(\!(\mathbb{Q})\!)$. More precisely, we can start with any perfect field
    $l$, e.g., $l=\mathbb{F}_p$, and consider the field $k=l(\!(\bigoplus_{i \in \mathbb{Z}}\mathbb{Q})\!)$. 
    We claim that $k \equiv k(\!(\mathbb{Q})\!)$.
    Let $\nu$ denote the Hahn series valuation on $k$ with residue field 
    $l$ and value group $\bigoplus_{i \in \mathbb{Z}}\mathbb{Q}$, and let 
    $w$ be the composition of the
    Hahn series valuation $v$ on $k(\!(\mathbb{Q})\!)$ and the valuation $\nu$ on $k$.
    Then, both $(k,\nu)$ and $(k(\!(\mathbb{Q})\!),w)$ are tame valued fields of characteristic $p$
    with divisible value group and the same perfect residue field $l$.
    Thus, the Ax-Kochen/Ershov Theorem for tame fields \cite[Theorem 7.1]{kuhlmann2016algebra} implies
    $$(k,\nu) \equiv (k(\!(\mathbb{Q})\!),w)$$
    as valued fields. 
    In particular, we conclude that $k \equiv k(\!(\mathbb{Q})\!)$ holds in the language of rings.

    Then $k$ is $t$-henselian of divisible tame type, and the valued field
    $(k(\!(\mathbb{Q})\!),v)$ is perfectoid. By Proposition \ref{prop:notdef}, $v$ is not definable.
    Moreover, if $(K,u)$ is any untilt of $(k(\!(\mathbb{Q})\!),v)$, then $u$ is also
    not definable by Proposition \ref{prop:notdef}.
\end{examp}

While the example above provided an explicit example of a perfectoid field of positive 
characteristic where the valuation was not $\Lring$-definable (even with parameters), it was less explicit in
mixed characteristic. Thus, we now also give an explicit construction in mixed characteristic.

\begin{examp} Take once again a non-algebraically closed 
field $k$ of positive characteristic with $k \equiv k(\!(\mathbb{Q})\!)$.
     An explicit example of a tame field
    of mixed characteristic with divisible value group and residue field $k$ can be constructed
    as follows: start with the ring of Witt vectors $W[k(\!(\mathbb{Q})\!)]$, take its
    fraction field $K_0 = W[k(\!(\mathbb{Q})\!)][1/p]$, 
    and let $v_0$ 
    be the Witt vector valuation on $K_0$. Let $(K_1,v_1)$ be a maximal totally ramified algebraic extension of $(K_0,v_0)$, 
    and let $v^\mathrm{alg}$ be the unique extension of $v_1$ to a fixed algebraic closure $K^\mathrm{alg}$.
    Let $R$ denote the ramification subgroup corresponding to $v^\mathrm{alg}/v_1$ of
    the absolute Galois group $\mathrm{Gal}(K_1^\mathrm{alg}/K_1)$ 
    (see \cite[Theorem 5.2.7]{engler2005valued}). 
    Let $C$ be any group theoretic complement of $R$ in $\mathrm{Gal}(K_1^\mathrm{alg}/K_1)$ (such exist by \cite[Theorem 2.2]{KPR}), and let $K$ be a fixed field of $C$. 
    Writing $v$ for the
    restriction of $v^\mathrm{alg}$ to $K$, we get that $(K,v)$ is a tame field of mixed characteristic 
    with value group $\mathbb{Q}$
    and residue field $k$ \cite[Proposition 4.5]{KPR}. Its completion
    $\widehat{(K,v)}$ is a perfectoid field of mixed characteristic such that $v$
    is not $\Lring$-definable.
\end{examp}

\subsection*{Main Theorem}
We are now in a position to prove our main theorem:
\begin{thm}
	\label{main-thm}
	Let \((K,v)\) be perfectoid. Then \(v\) is definable if and only if
    $v$ is $\emptyset$-definable if and only if at least one of the following holds
	\begin{enumerate}
		\item \((K,v)\) is not defectless, 
		\item \(vK\) is not divisible, or
		\item \(Kv\) is not t-henselian of divisible-tame type.
	\end{enumerate}
\end{thm}
\begin{proof}
    If \((K,v)\) is not defectless, then \(v\) is $\emptyset$-definable by \cref{def-from-defect}.
    If \(vK\) is not divisible, then \(v\) is $\emptyset$-definable by \cref{hong}.
    If \(Kv\) is not t-henselian of divisible-tame type, then \(v\) is 
    $\emptyset$-definable by \cref{prop:defnotdtt}.

    Conversely, assume that \(\neg (1)\) and \(\neg (2)\) and \(\neg (3)\) hold.
    Since \((K,v)\) is perfectoid, we know that it is henselian and \(Kv\) is perfect. 
    Together with \(\neg (1)\), \((K,v)\) is defectless, and \(\neg (2)\), \(vK\) is divisible, we get that \((K,v)\) is tame with divisible value group.
    Now, by \(\neg (3)\), \(Kv\) is t-henselian of divisible-tame type. 
    Thus, by \cref{prop:notdef}, \(v\) cannot be definable (even with parameters).
\end{proof}

Recall that, by \cref{fact:perfectoid}, a perfectoid field \((K,v)\) and its tilt \((K^\flat,v^\flat)\) have the same residue field and value group, and one of them is defectless if and only if the other is.
Thus we get the following corollary.

\begin{kor}
	\label{kor:v-def--vflat-def}
	Let \((K,v)\) be a perfectoid field of mixed characteristic. 
	Then, \(v\) is $(\emptyset$-$)$definable if and only if \(v^\flat\) is $(\emptyset$-$)$definable.
\end{kor}

On closer inspection one can see that none of 
\cref{hong}, \cref{prop:defnotdtt}, \cref{thm:defect} or \cref{prop:notdef} required the 
field $(K,v)$ to be complete. 
In each case, the assumptions applied to any rank-1 henselian valued field $(K,v)$ of 
residue characteristic $p>0$ such that $K$ is perfect (if $\chara{K}=p)$ or $\O_v/p$ is
semi-perfect (otherwise). However, by \cite[Theorem 5.1.4]{jahnke-kartas2023beyond}, any
such field $(K,v)$ is an $\Lval$-elementary substructure of its completion 
$\widehat{({K},{v})}=(\widehat{K},\widehat{v})$, and the latter is a perfectoid field. Thus, the $\emptyset$-definability 
(respectively definability) of $v$ and $\widehat{v}$ are in fact equivalent.

\section{Producing and avoiding uniformity}
One might hope to ask whether one can use the same formula when defining $v$ and $v^\flat$.
In the explicit definition we have encountered when defining a valuation using a non-divisibility in the value group (see \cref{hong}), in a natural way the same 
formulae defined $v$ and $v^\flat$. However, due to the fact that characteristic $p$ and characteristic $0$ \label{uniform}
can be differentiated by an $\mathcal{L}_\mathrm{ring}$-formula, we can always artificially 
create or dispel uniformity, as illustrated in the next proposition:
\begin{prop}
	Let \((K,v)\) be perfectoid of mixed characteristic.
	Assume that \(v\) is definable.
	Then,
	\begin{enumerate}
		\item there is an \(\Lring\)-formula \(\varphi(x,\vec{y})\) such that \(\O_v =\varphi(K,\vec{c})\) for some parameter tuple \(\vec{c}\in K^n\) but for any choice of \(\vec{c}'\in (K^\flat)^n\) we have \(\O_{v^\flat}\neq \varphi(K^\flat,\vec{c}')\);
		\item there is an \(\Lring\)-formula \(\varphi(x,\vec{y})\) such that \(\O_{v^\flat} =\varphi(K^\flat,\vec{c}')\) for some parameter tuple \(\vec{c}'\in (K^\flat)^n\) but for any choice of \(\vec{c}\in K^n\) we have \(\O_v\neq \varphi(K,\vec{c})\);
		\item there is an \(\Lring\)-formula \(\varphi(x,\vec{y})\) such that there are parameter tuples \(\vec{c}\in K^n\) and \(\vec{c}\in (K^\flat)^n\) with \(\O_v=\varphi(K,\vec{c})\) and \(\O_{v^\flat}=\varphi(K^\flat,\vec{c}')\).
	\end{enumerate}
\end{prop}
\begin{proof}
	Note that by \cref{kor:v-def--vflat-def}, also \(v^\flat\) is definable. Let \(p=\chara(Kv)=\chara(K^\flat)\) and let \(\chi_p:\equiv \underbrace{1+1+\ldots+1}_{p\text{-many}}=0\).
	\begin{enumerate}
		\item Let \(\psi(x,\vec{y})\) be some \(\Lring\)-formula defining \(v\) in \(K\), i.e., there is a parameter tuple \(\vec{c}\in K^n\) such that \(\O_v=\psi(K,\vec{c})\).
		Let \(\varphi(x,\vec{y}):\equiv \neg \chi_p \wedge \psi(x,\vec{y})\). 
		Since \(K\) has characteristic zero, \(\neg\chi_p\) holds true in \(K\), so \(\O_v=\psi(K,\vec{c})=\varphi(K,\vec{c})\).
		
		On the other hand, as \(K^\flat\) has characteristic \(p\), \(\neg\chi_p\) is false in \(K^\flat\). 
		Thus, for any parameter tuple \(\vec{c}'\in (K^\flat)^n\), we get 
		\(\O_{v^\flat}\neq \emptyset=\varphi(K^\flat,\vec{c}')\).
		\item Similar to (1), use \(\chi_p\) instead of \(\neg\chi_p\).
		\item Let \(\psi_1(x,\vec{y})\) be some \(\Lring\)-formula defining \(v\) in \(K\), i.e., there is a parameter tuple \(\vec{c}\in K^n\) such that \(\O_v=\psi_1(K,\vec{c})\).
		Let \(\psi_2(x,\vec{y})\) be some \(\Lring\)-formula defining \(v^\flat\) in \(K^\flat\), i.e., there is a parameter tuple \(\vec{c}'\in (K^\flat)^n\) such that \(\O_{v^\flat}=\psi_2(K^\flat,\vec{c}')\).
		We may assume that both parameter tuples have the same length as we can always extend it by unused parameters.
		Then,
		\[
		\varphi(x,\vec{y}):\equiv (\neg\chi_p \wedge \psi_1(x,\vec{y})) \vee (\chi_p \wedge \psi_2(x,\vec{y}))
		\]
		defines both \(v\) in \(K\) and \(v^\flat\) in \(K^\flat\). Indeed, as \(\neg\chi_p\) holds in \(K\), we have \(\O_v=\psi_1(K,\vec{c})=\varphi(K,\vec{c})\), and as \(\chi_p\) holds in \(K^\flat\), we have \(\O_{v^\flat}=\psi_2(K^\flat,\vec{c}')=\varphi(K^\flat,\vec{c}')\).\qedhere
	\end{enumerate}
\end{proof}

\section{Quantifier complexity}
The previous section illustrates that asking for a uniform definition for a perfectoid
valuation and its tilt has no merit. \label{sec:qf}
In this section, we focus on the quantifier
complexity of such definitions instead. We show that $\exists$-$\emptyset$-definability
and  $\forall$-$\emptyset$-definability both untilt, and that
$\forall$-$\emptyset$-definability does not tilt. The question whether 
$\exists$-$\emptyset$-definability tilts remains open.

Our approach relies on work by Anscombe and Fehm \cite{anscombe-fehm2017characterizing}, and 
we begin by recalling their definitions. They phrase all their definitions and statements in terms of
$C$-fields. Here, we only consider the special case where $C=\mathbb{Z}.$

\begin{defi}[{see \cite[Definition 3.5 and Lemma 3.7]{anscombe-fehm2017characterizing}}]
   Let $K$ be a field.  \label{embres}
     \begin{enumerate}
\item We say that $K$ has embedded residue if there is some
    $K' \equiv K$ (in $\Lring$) and a nontrivial valuation $v$ on $K'$ such that there is an embedding of $Kv$ into $K'$.
\item We say that $K$ is $\mathbb{Z}$-large if there is some $K' \equiv K$ (in $\Lring$),
a subfield $E\subseteq K'$ and a nontrivial henselian valuation $v$ on $E$
such that $Ev$ is isomorphic to $K'$.
     \end{enumerate}
\end{defi}

In fact, by \cite[Lemma 3.3]{anscombe2026noteexistentiallythenselianfields}, a field $K$ is $\mathbb{Z}$-large if and only if it has the same
existential $\Lring$-theory as some field admitting a nontrivial henselian
valuation. Any large field in the sense of Pop is $\mathbb{Z}$-large.
Anscombe and Fehm show that not having embedded residue (respectively not being $\mathbb{Z}$-large) can be used to characterize which henselian valuations are 
$\exists$-$\emptyset$-definable (respectively $\forall$-$\emptyset$-definable):

\begin{fact}[{\cite[Corollary 5.3]{anscombe-fehm2017characterizing}}]
For any field $F$, the following are equivalent. \label{AF}
\begin{enumerate}
\item[$(0^{\exists})$] The valuation ring is $\exists$-$\emptyset$-definable in \emph{some} nontrivially henselian valued 
field $(K,v)$ with $\chara{K}=\chara{Kv}$ and $Kv\equiv_{\Lring} F$.
\item[$(1^{\exists})$] There is a $\exists$-$\emptyset$-$\Lring$-formula 
$\psi_F(x)$ that defines 
the valuation ring in any henselian valued field $(K,v)$ with $Kv\equiv_{\Lring} F$.
\item[$(2^{\exists})$] $F$ does not have embedded residue.
\end{enumerate}
Also, the following are equivalent.
\begin{enumerate}
\item[$(0^{\forall})$] The valuation ring is $\forall$-$\emptyset$-definable in 
\emph{some} nontrivially henselian valued 
field $(K,v)$ with $\chara{K}=\chara{Kv}$ and $Kv\equiv_{\Lring} F$.
\item[$(1^{\forall})$]
There is a $\forall$-$\emptyset$-$\Lring$-formula 
$\psi_F(x)$ that defines 
the valuation ring in any henselian valued field $(K,v)$ with $Kv\equiv_{\Lring} F$.
\item[$(2^{\forall})$] $F$ is not $\mathbb{Z}$-large. 
\end{enumerate}
\end{fact}

As an immediate consequence, we get that $\exists$-$\emptyset$-definability and
$\forall$-$\emptyset$-definability untilt:
\begin{kor} \label{untilt}
    Let $(K,v)$ be perfectoid with tilt $(K^\flat,v^\flat)$. If $v^\flat$ is
    definable by an $\exists$-$\emptyset$-formula (respectively
$\forall$-$\emptyset$-formula) in $\Lring$, then so is $v$.
\end{kor}
\begin{proof}
    As $Kv=K^\flat v^\flat$, this follows from \cref{AF} $(0^{\exists}) \implies (1^{\exists})$
    (respectively $(0^{\forall}) \implies (1^{\forall})$).
\end{proof}

For $\forall$-$\emptyset$-definability, we can use \cref{hong} to show that it does
not tilt.

\begin{examp} \label{nottilt}
Let $(K,v)=\widehat{(\mathbb{Q}_p(\mu_{\infty}),v_p)}$. Its residue field is the algebraically closed field $\mathbb{F}_p^{\mathrm{alg}}$. Let $q\neq p$ be a prime.
As $q \nmid v_p(p)$, the valuation $v_p$ is 
$\forall$-$p$-definable by \cref{hong}, and -- as $p$ can be replaced by ${1+\dots +1}$ ($p$ times) -- even $\forall$-$\emptyset$-definable.

However, $\mathbb{F}_p^{\mathrm{alg}}$ is $\mathbb{Z}$-large as it is algebraically closed 
(\cite[Lemma 3.8]{anscombe-fehm2017characterizing}). 
Thus, by \cref{AF} $(0^{\forall}) \implies (2^{\forall})$, 
in $(K^\flat, v^\flat)$ the valuation $v^\flat$ is not $\forall$-$\emptyset$-definable. We conclude that $\forall$-$\emptyset$-definability is not preserved under tilting.
\end{examp}

If we could show that rank-1 non-divisible value groups also give rise
to $\exists$-$t$-definability of the valuation akin to \cref{hong}, we could use a variant of the same argument to see that $\exists$-$\emptyset$-definability
does not tilt. However, the existential analogue of Hong's Theorem no longer holds.


Finding an example that illustrates that $\exists$-$\emptyset$-definability does not tilt
could also give a negative answer to 
\cite[Question 6.25]{anscombe-fehm2017characterizing}.

\section{The classical case}
We now take a closer look at the \emph{classical case}, i.e., when 
\((K,v)\) is perfectoid and \(Kv\) is an algebraic extension of \(\F_p\).
For example, this happens if $(K,v)$ is
obtained as
the completion of an algebraic extension 
of one of $(\mathbb{Q}_p,v_p)$, $(\mathbb{Q},v_p)$, $(\mathbb{F}_p(\!(t)\!),v_t)$ or
$(\mathbb{F}_p(t),v_t)$. \label{sec:classical}

We prove that in this setting, 
the perfectoid valuation is $\emptyset$-definable unless \(K\) is algebraically closed.
Moreover, in this case, we always get an 
$\exists$-$\emptyset$-definition, unless $Kv$ is
algebraically closed.

We start with a simple lemma:
\begin{lem}
\label{alg-Fp+t-hens->alg-closed}
    Let \(F\) be an algebraic extension of \(\F_p\). 
    Then the following are equivalent: 
	\begin{enumerate}
		\item \(F\) is t-henselian of divisible tame type.
		\item \(F\) is t-henselian.
		\item \(F=\F_p\alg\).
	\end{enumerate}
\end{lem}
\begin{proof}
	Clearly, (1) implies (2).
	
	To prove that (2) implies (3), assume \(F/\F_p\) is algebraic. Then $F$ is either finite or pseudo-algebraically closed (PAC)
    \cite[Corollary 11.2.4]{fried2008field}. If $F$ is finite, so is any $L \equiv F$.
    Since finite fields admit no nontrivial valuations, $L$ cannot admit a nontrivial henselian valuation, so $F$ is not $t$-henselian.
    On the other hand, if $F$ is PAC, so is any $L \equiv F$ \cite[Proposition 11.3.2]{fried2008field}.
	By \cite[Corollary~11.5.5]{fried2008field}, any field which is both PAC and henselian must be separably closed.
	Thus \(L\), and by elementary equivalence also \(F\), is separably closed, so \(F=\F_p\alg\).
	
	For (3) implies (1), see \cref{ex:alg-cl->t-hens-dtt}.
\end{proof}

\begin{prop}
Let $(K,v)$ be perfectoid and assume that $Kv$ is an algebraic extension of 
$\mathbb{F}_p$. If $K$ is not algebraically closed, then $v$ is $\emptyset$-definable.   \label{prop:classical}
\end{prop}
\begin{proof}
    Recall that 
    a defectless henselian valued field with algebraically closed residue field
    and divisible value group is algebraically closed. Thus, assuming $K$ is not algebraically
    closed, 
    either $(K,v)$ has defect, or $vK$ is no divisible, or $Kv$ is not algebraically 
    closed. In the first two cases, $v$ is $\emptyset$-definable by 
    \cref{def-from-defect}  and
    \cref{hong} respectively. In the third case, we conclude that $v$ is $\emptyset$-definable by combining \cref{prop:defnotdtt} with \cref{alg-Fp+t-hens->alg-closed}. 
\end{proof}

Finally, we not that we get $\exists$-$\emptyset$-definability in the classical
case, unless $Kv$ is algebraically closed. This is shown in \cite[Corollary 6.2]{anscombe-fehm2017characterizing}:
\begin{fact}
        Let $(K,v)$ be henselian and assume that $Kv\neq Kv^\mathrm{alg}$ 
    is an algebraic extension of $\mathbb{F}_p$. Then $v$ is $\exists$-$\emptyset$-definable. \label{fact:classical}
\end{fact}

We cannot hope to obtain the same for $\forall$-$\emptyset$-definability, as the 
following example shows. 
This also gives us another example that $\forall$-$\emptyset$-definability does not tilt.
\begin{examp}
    Let $F \neq F^\mathrm{alg}$ be an infinite algebraic extension of $\F_p$. Consider
    the perfectoid field $(F(\!(\mathbb{Z}[1/p])\!),v)$. Then $F$ is large, and so 
    $v$ is not $\forall$-$\emptyset$-definable by 
    \cite[Corollary 6.21]{anscombe-fehm2017characterizing}. However, in every untilt $(K,u)$ of $(F(\!(\mathbb{Z}[1/p])\!),v)$, the valuation $u$ is $\forall$-$p$-definable: in each such untilt, we have 
    $q\nmid u(p)$ for some prime $q \neq p$ (note that no non-zero element of $\mathbb{Z}[1/p]$ is divisible by all primes $q$)
    and so we can apply \cref{hong}. Since $p$ is a $\emptyset$-definable
    constant, $u$ is in fact $\forall$-$\emptyset$-definable.
\end{examp}

\addcontentsline{toc}{chapter}{{Bibliography}}

\bibliographystyle{amsalpha}

\bibliography{bibfile}

@article{anscombe-fehm2017characterizing,
	title={Characterizing diophantine henselian valuation rings and valuation ideals},
	author={Anscombe, Sylvy and Fehm, Arno},
	journal={Proceedings of the London Mathematical Society},
	volume={115},
	number={2},
	pages={293--322},
	year={2017},
	publisher={Wiley Online Library}
}

@article{anscombe2018henselianity,
	title={Henselianity in the language of rings},
	author={Anscombe, Sylvy and Jahnke, Franziska},
	longjournal={Annals of Pure and Applied Logic},
	journal={Ann. Pure Appl. Log.},
	volume={169},
	number={9},
	pages={872--895},
	year={2018},
	publisher={Elsevier}
}

@article{anscombe2024characterizing,
	author = {Anscombe, Sylvy and Jahnke, Franziska},
	title = {Characterizing {NIP} henselian fields},
	journal = {J. Lond. Math. Soc.},
	longjournal = {Journal of the London Mathematical Society},
	volume = {109},
	number = {3},
	pages = {e12868},
	doi = {https://doi.org/10.1112/jlms.12868},
	url = {https://londmathsoc.onlinelibrary.wiley.com/doi/abs/10.1112/jlms.12868},
	eprint = {https://londmathsoc.onlinelibrary.wiley.com/doi/pdf/10.1112/jlms.12868},
	year = {2024}
}

@article{anscombe2026noteexistentiallythenselianfields,
	title={A note on existentially t-henselian fields}, 
	author={Sylvy Anscombe},
	year={2026},
	eprint={2603.27612},
	archivePrefix={arXiv},
	primaryClass={math.LO},
	url={https://arxiv.org/abs/2603.27612}, 
	journal={arXiv preprint \href{https://arxiv.org/abs/2603.27612}{arXiv:2603.27612}},
}

@article{Ax,
    AUTHOR = {Ax, James},
     TITLE = {On the undecidability of power series fields},
   JOURNAL = {Proc. Amer. Math. Soc.},
  FJOURNAL = {Proceedings of the American Mathematical Society},
    VOLUME = {16},
      YEAR = {1965},
     PAGES = {846},
      ISSN = {0002-9939,1088-6826},
   MRCLASS = {02.74 (10.80)},
  MRNUMBER = {177890},
       DOI = {10.2307/2033940},
       URL = {https://doi.org/10.2307/2033940},
}

@incollection{Chatzidakis2026,
author="Chatzidakis, Zo{\'e}",
editor="Haskell, Deirdre",
title="Notes on the Model Theory of Finite and Pseudo-Finite Fields",
bookTitle="Model Theory: Selected Lectures from the 2021 Thematic Program",
year="2026",
publisher="Springer Nature Switzerland",
address="Cham",
pages="1--43",
isbn="978-3-032-15023-3",
doi="10.1007/978-3-032-15023-3_1",
url="https://doi.org/10.1007/978-3-032-15023-3_1"
}

@book{efrat2006valuations,
	title={Valuations, orderings, and {Milnor} {$K$-theory}},
	author={Efrat, Ido},
	number={124},
	year={2006},
	publisher={Amer. Math. Soc.}
}

@book{engler2005valued,
	title={Valued fields},
	author={Engler, Antonio J and Prestel, Alexander},
	year={2005},
	publisher={Springer Science \& Business Media}
}

@book{fried2008field,
  title={Field Arithmetic},
  author={Fried, Michael D. and Jarden, Moshe},
  longjournal={Ergebnisse der Mathematik und ihrer Grenzgebiete},
  journal={Ergeb. Math. Grenzgeb.},
  volume={11},
  edition={3},
  year={2008},
  publisher={Springer}
}

@book{hodges1993model,
	title={Model theory},
	author={Hodges, Wilfrid},
	year={1993},
	publisher={Cambridge University Press}
}

@book{hodges1997shorter,
	title={A shorter model theory},
	author={Hodges, Wilfrid},
	year={1997},
	publisher={Cambridge University Press}
}

@article{hong2014definable,
	title={Definable non-divisible Henselian valuations},
	author={Hong, Jizhan},
	journal={Bull. Lond. Math. Soc.},
	volume = {46},
	number = {1},
	pages = {14-18},
	doi = {https://doi.org/10.1112/blms/bdt074},
	url = {https://londmathsoc.onlinelibrary.wiley.com/doi/abs/10.1112/blms/bdt074},
	eprint = {https://londmathsoc.onlinelibrary.wiley.com/doi/pdf/10.1112/blms/bdt074},
	year = {2014}
}

@article{jahnke-kartas2023beyond,
    AUTHOR = {Jahnke, Franziska and Kartas, Konstantinos},
     TITLE = {Beyond the {F}ontaine-{W}intenberger theorem},
   JOURNAL = {J. Amer. Math. Soc.},
  FJOURNAL = {Journal of the American Mathematical Society},
    VOLUME = {38},
      YEAR = {2025},
    NUMBER = {4},
     PAGES = {997--1047},
      ISSN = {0894-0347,1088-6834},
   MRCLASS = {12L12 (11U09 12J20 13A18 14G45)},
  MRNUMBER = {4930328},
       DOI = {10.1090/jams/1056},
       URL = {https://doi.org/10.1090/jams/1056},
}

@article{JK14a,
author={Jahnke, Franziska and Koenigsmann, Jochen},
title = {Definable henselian valuations},
longjournal = {The Journal of Symbolic Logic},
journal = {J. Symb. Log.},
volume = {80},
issue = {01},
year = {2015},
issn = {1943-5886},
pages = {85--99},
numpages = {15},
doi = {10.1017/jsl.2014.64},
URL = {http://journals.cambridge.org/article_S0022481214000644},
}

@article {andre,
    AUTHOR = {Andr\'e, Yves},
     TITLE = {La conjecture du facteur direct},
   JOURNAL = {Publ. Math. Inst. Hautes \'Etudes Sci.},
  FJOURNAL = {Publications Math\'ematiques. Institut de Hautes \'Etudes
              Scientifiques},
    VOLUME = {127},
      YEAR = {2018},
     PAGES = {71--93},
      ISSN = {0073-8301,1618-1913},
   MRCLASS = {13D22 (13A35 13B40 13D09 18A99)},
  MRNUMBER = {3814651.},
MRREVIEWER = {Marcel\ Morales},
       DOI = {10.1007/s10240-017-0097-9},
       URL = {https://doi.org/10.1007/s10240-017-0097-9},
}

@article{JK15,
	Author = {Franziska Jahnke and Jochen Koenigsmann},
	Doi = {http://dx.doi.org/10.1016/j.apal.2015.03.003},
	Issn = {0168-0072},
	Longjournal = {Annals of Pure and Applied Logic},
	Journal = {Ann. Pure Appl. Log.},
	Number = {7--8},
	Pages = {741 - 754},
	Title = {Uniformly defining $p$-henselian valuations},
	Url = {http://www.sciencedirect.com/science/article/pii/S0168007215000238},
	Volume = {166},
	Year = {2015}}

@article{kartas2024decidability,
	title={Decidability via the tilting correspondence},
	author={Kartas, Konstantinos},
	journal={Algebra Number Theory},
	longjournal={Algebra \& Number Theory},
	volume={18},
	number={2},
	pages={209--248},
	year={2024},
	publisher={Mathematical Sciences Publishers}
}

@article{ketelsen2024definable,
	title={Definable Henselian valuations in positive residue characteristic},
	author={Ketelsen, Margarete and Ramello, Simone and Szewczyk, Piotr},
	journal={to appear in J. Symb. Log.},
	year={published online 2024}
}

@article{Koe95,
	Author = {Koenigsmann, Jochen},
	Fjournal = {Manuscripta Mathematica},
	Journal = {Manuscripta Math.},
	Number = {1},
	Pages = {89--99},
	Title = {{$p$}-{H}enselian fields},
	Volume = {87},
	Year = {1995}}

@article{Koehab,
	AUTHOR = {Koenigsmann, Jochen},
	TITLE = {Elementary characterization of fields by their absolute
	{G}alois group},
	JOURNAL = {Siberian Adv. Math.},
	FJOURNAL = {Siberian Advances in Mathematics},
	VOLUME = {14},
	YEAR = {2004},
	NUMBER = {3},
	PAGES = {16--42},
}

@article{KPR,
	Author = {Kuhlmann, Franz-Viktor and Pank, Matthias and Roquette, Peter},
	Fjournal = {Manuscripta mathematica},
	Journal = {Manuscripta Math.},
	Pages = {39-68},
	Title = {Immediate and Purely Wild Extensions of Valued Fields.},
	Url = {http://eudml.org/doc/155126},
	Volume = {55},
	Year = {1986},
}

@article{kuhlmann2016algebra,
	title={The algebra and model theory of tame valued fields},
	author={Kuhlmann, Franz-Viktor},
	journal = {J. Reine Angew. Math.},
	longjournal = {Journal f{\"u}r die reine und angewandte Mathematik (Crelles Journal)},
	volume={2016},
	number={719},
	pages={1--43},
	year={2016},
	publisher={De Gruyter}
}

@article{kuhlmann-rzepka2023valuation,
	title={The valuation theory of deeply ramified fields and its connection with defect extensions}, 
	author={Franz-Viktor Kuhlmann and Anna Rzepka},
	journal={Trans. Amer. Math. Soc.},
	volume={376},
	number={4},
	pages={2693--2738},
	year={2023},
	publisher={American Mathematical Society},
}

@Article{FontWint,
 Author = {Fontaine, Jean-Marc and Wintenberger, Jean-Pierre},
 Title = {Extensions alg{\'e}\-briques et corps des normes des extensions {APF} des corps locaux},
 FJournal = {Comptes Rendus Hebdomadaires des S{\'e}ances de l'Acad{\'e}mie des Sciences, S{\'e}rie A},
 Journal = {C. R. Acad. Sci., Paris, S{\'e}r. A},
 ISSN = {0366-6034},
 Volume = {288},
 Pages = {441--444},
 Year = {1979},
 zbMATH = {3625519},
 Zbl = {0403.12018}
}

@incollection{FehmJahnke,
    AUTHOR = {Fehm, Arno and Jahnke, Franziska},
     TITLE = {Recent progress on definability of {H}enselian valuations},
 BOOKTITLE = {Ordered algebraic structures and related topics},
    SERIES = {Contemp. Math.},
    VOLUME = {697},
     PAGES = {135--143},
 PUBLISHER = {Amer. Math. Soc., Providence, RI},
      YEAR = {2017},
      ISBN = {978-1-4704-2966-9; 978-1-4704-4222-4},
   MRCLASS = {12J10 (03C60 12L12)},
  MRNUMBER = {3716069.},
MRREVIEWER = {Franz-Viktor\ Kuhlmann},
       DOI = {10.1090/conm/697/14049},
       URL = {https://doi.org/10.1090/conm/697/14049},
}

@article{kartas2026perfectoidcitransfer,
      title={Perfectoid {$C_i$} transfer}, 
      author={Konstantinos Kartas},
      year={2026},
      eprint={2504.14719},
      archivePrefix={arXiv},
      primaryClass={math.NT},
      url={https://arxiv.org/abs/2504.14719}, 
      	journal={arXiv preprint \href{https://arxiv.org/abs/2504.14719}{arXiv:2504.14719}},
}

@article{kuhlmann2025modeltheorytamevalued,
      title={Model theory of tame valued fields and beyond: recent developments and open questions}, 
      author={Franz-Viktor Kuhlmann},
      year={2025},
      eprint={2512.06386},
      archivePrefix={arXiv},
      primaryClass={math.LO},
      url={https://arxiv.org/abs/2512.06386}, 
      journal={arXiv preprint \href{https://arxiv.org/abs/2512.06386}{arXiv:2512.06386}},
}

@article{gambardella2026transferprincipleskatokuzumakiconjecture,
      title={Transfer principles and the {K}ato-{K}uzumaki conjecture}, 
      author={Felipe Gambardella and Konstantinos Kartas},
      year={2026},      
      eprint={2603.01815},
      archivePrefix={arXiv},
      primaryClass={math.NT},
      url={https://arxiv.org/abs/2603.01815}, 
    journal={arXiv preprint \href{https://arxiv.org/abs/2603.01815}{arXiv:2603.01815}},
}

@misc{lurie2018fargues,
	title={The {F}argues-{F}ontaine Curve},
	author={Lurie, Jacob},
	note={Lecture Notes, available at \url{https://www.math.ias.edu/~lurie/205.html}},
	year={2018}
}

@article{Pralg,
	AUTHOR = {Prestel, Alexander},
	TITLE = {Algebraic number fields elementarily determined by their
	absolute {G}alois group},
	JOURNAL = {Israel J. Math.},
	FJOURNAL = {Israel Journal of Mathematics},
	VOLUME = {73},
	YEAR = {1991},
	NUMBER = {2},
	PAGES = {199--205},
}

@article{prestel-ziegler1978model,
	title={Model theoretic methods in the theory of topological fields.},
	author={Prestel, Alexander and Ziegler, Martin},
	year={1978},
	pages = {318--341},
	volume = {1978},
	number = {299-300},
	journal = {J. Reine Angew. Math.},
	longjournal = {Journal für die Reine und Angewandte Mathematik. [Crelle's Journal]},
	doi = {doi:10.1515/crll.1978.299-300.318},
	publisher={Walter de Gruyter, Berlin/New York Berlin, New York}
}

@article{rideaukikuchi2025tiltingequivalencebiinterpretation,
      title={The tilting equivalence as a bi-interpretation}, 
      author={Silvain Rideau-Kikuchi and Thomas Scanlon and Pierre Simon},
      year={2025},
      eprint={2505.01321},
      archivePrefix={arXiv},
      primaryClass={math.LO},
      url={https://arxiv.org/abs/2505.01321}, 
    journal={arXiv preprint \href{https://arxiv.org/abs/2505.01321}{arXiv:2505.01321}},
}

@article{scholze2012perfectoid,
	title={Perfectoid spaces},
	author={Scholze, Peter},
	journal={Publ. math. IH{\'E}S},
	volume={116},
	number={1},
	pages={245--313},
	year={2012},
	publisher={Springer}
}

@article {Wint,
	AUTHOR = {Wintenberger, Jean-Pierre},
	TITLE = {Le corps des normes de certaines extensions infinies de corps
	locaux; applications},
	JOURNAL = {Ann. Sci. \'Ecole Norm. Sup. (4)},
	FJOURNAL = {Annales Scientifiques de l'\'Ecole Normale Sup\'erieure.
	Quatri\`eme S\'erie},
	VOLUME = {16},
	YEAR = {1983},
	NUMBER = {1},
	PAGES = {59--89},
	ISSN = {0012-9593},
	MRCLASS = {11S20 (11S15)},
	MRNUMBER = {719763.},
	MRREVIEWER = {Antonio\ Jos\'e\ Engler},
	URL = {http://www.numdam.org/item?id=ASENS_1983_4_16_1_59_0},
}

\end{document}